\documentclass{amsart}

\DeclareMathSymbol{\twoheadrightarrow}  {\mathrel}{AMSa}{"10}
        \def\GG{{\mathcal G}}

\def\Q{{\mathbb Q}}
\def\Z{{\mathbb Z}}
\def\C{{\mathbb C}}
\def\CC{{\mathfrak C}}
                         
                   \def\sp{{\mathfrak{sp}}}
                   \def\sL{{\mathfrak{sl}}}
\def\RR{{\mathbb R}}
\def\F{{\mathbb F}}
\def\P{{\mathfrak P}}

                                    \def\XX{{\mathbf X}}
                                     \def\YY{{\mathbf Y}}
                                     \def\ZZ{{\mathbf Z}}

\def\f{{\tilde F}}
                     \def\f0{{\mathfrak f}}

             \def\K{\mathrm{K}}

\def\A8{{\mathbf A}_8}
\def\Alt{\mathrm{Alt}}

\def\RR{{\mathfrak R}}
\def\Perm{\mathrm{Perm}}
\def\Gal{\mathrm{Gal}}
\def\Pic{\mathrm{Pic}}

\def\End{\mathrm{End}}
\def\Aut{\mathrm{Aut}}

\def\I{{\mathrm{Id}}}
                                           \def\ss{{\mathrm{ss}}}
                                            \def\Proj{\mathrm{Proj}}

                              \def\rest{{\mathrm{rest}}}
\def\II{{{\mathcal I}}}

\def\ST{{\mathbf S}}

\def\fchar{\mathrm{char}}

\def\Sp{\mathrm{Sp}}
            \def\Gp{\mathrm{Gp}}
               
                                                                  \def\Spec{\mathrm{Spec}}

\def\A{\mathbf{A}}
                  \def\X{\mathcal{X}}
\def\T{{\mathrm T}}

\def\dim{\mathrm{dim}}
\def\Oc{{\mathcal O}}

       \def\PGL{\mathrm{PGL}}

      \def\g{{\mathfrak g}}

            \def\m{{\mathfrak m}}

                                   \def\gp{\mathfrak{gp}}
                                   \def\sp{\mathfrak{sp}}
                                     \def\sL{\mathfrak{sl}}

                         \def\Cc{\mathcal{C}}
                         
                                               \def\PP{{\mathbf P}}
\newtheorem{thm}{Theorem}[section]
\newtheorem{lem}[thm]{Lemma}
\newtheorem{cor}[thm]{Corollary}

\theoremstyle{definition}
\newtheorem{defn}[thm]{Definition}
\newtheorem{ex}[thm]{Example}
\newtheorem{rem}[thm]{Remark}

\newtheorem{rems}[thm]{Remarks}
        \newtheorem{sect}[thm]{}
\title[Families of hyperelliptic jacobians]{Two-dimensional  families of  hyperelliptic jacobians with big monodromy}

\author[Yuri G.\ Zarhin]{Yuri G.\ Zarhin}
\thanks{This work was partially supported by a grant from the Simons Foundation (\#246625 to Yuri Zarkhin).}

\address{Department of Mathematics, Pennsylvania State University,
University Park, PA 16802, USA}
\address{Department of Mathematics, The Weizmann Institute of Science,
 P.O.B. 26,  Rehovot 7610001, Israel}

\email{zarhin\char`\@math.psu.edu}
\begin{document}

\begin{abstract}
Let $K$ be a global field of characteristic different from $2$ and $u(x)\in K[x]$ be an irreducible polynomial of even degree $2g\ge 6$, whose Galois group over $K$ is either the full symmetric group $\ST_{2g}$ or the alternating group $\A_{2g}$.  We describe explicitly how to choose (infinitely many) pairs of distinct  $t_1, t_2 \in K$  such that
 the $g$-dimensional  jacobian of a hyperelliptic curve $y^2=(x-t_1)(x-t_2))u(x)$ has no
nontrivial endomorphisms over an algebraic closure of $K$ and has big monodromy.
\end{abstract}


\maketitle

\section{Statements}
\label{intro}

 As usual, $\Z$, $\Q$ and $\C$ stand for the ring of integers, the
field of rational numbers
 and the field of complex numbers
respectively. If $\ell$ is a prime then we write $\F_{\ell},\Z_{\ell}$ and
$\Q_{\ell}$ for the ${\ell}$-element (finite) field, the ring of ${\ell}$-adic
integers and field of ${\ell}$-adic numbers respectively. If $A$ is a finite
set then we write $\#(A)$ for the number of its elements.

If $C$ is a commutative ring with $1$,  $V$ a free $C$-module of finite rank and
$e: V \times V \to C$ an alternating $C$-bilinear form then we write
$$\Sp(V,e) \subset \Gp(V,e) \subset \Aut_C(V)$$
for the symplectic group
$$\Sp(V,e) =\{u \in \Aut_C(V)\mid e(ux,uy)=e(x,y) \ \forall \ x,y \in V\}$$
and the group of symplectic similitudes
$\Gp(V,e)$ that consists of all automorphisms $u$ of $V$ such that there exists a {\sl constant}
$c=c(u) \in C^{*}$ such that
$$ e(ux,uy)=c \cdot e(x,y) \ \forall \ x,y \in V.$$

Let $K$ be a field of characteristic different from $2$, let $\bar{K}$ be its
algebraic closure and $\Gal(K)=\Aut(\bar{K}/K)$ its absolute Galois group.
If $L \subset \bar{K}$ is a finite separable algebraic extensiuon of $K$ then $\bar{K}$ is an algebraic closure of $L$ and $\Gal(L)=\Aut(\bar{K}/L)$ is an open subgroup of finite index in $\Gal(K)$; actually, the index equals degree $[L:K]$ of the field extension $L/K$.

Let
$n\ge 5$ be an integer, $f(x)\in K[x]$  a degree $n$ polynomial {\sl without
multiple roots}, $\RR_f \subset \bar{K}$ the $n$-element set of its roots,
$K(\RR_f) \subset \bar{K}$ the splitting field of $f(x)$ and
$\Gal(f)=\Gal(K(\RR_f)/K)$ the Galois group of $f(x)$ over $K$. One may view
$\Gal(f)$ as a certain group of permutations of $\RR_f$. Let $C_f: y^2=f(x)$ be
the corresponding hyperelliptic curve of genus $\lfloor(n-1)/2\rfloor$. Let
$J(C_f)$ be the jacobian of $C_f$; it is a $\lfloor(n-1)/2\rfloor$-dimensional
abelian variety that is defined over $K$.

Let $X$ be an abelian variety that is defined over $K$.
We write $\End(X)$ for the ring
of all
 $\bar{K}$-endomorphisms of $X$. As usual, we write $\End^{0}(X)$ for the corresponding
 (finite-dimensional semisimple) $\Q$-algebra $\End(X)\otimes\Q$.

If $m$ is a positive integer that is not divisible by $\fchar(K)$ then we write $X_m$ for the kernel of multiplication by $m$ in $X(\bar{K})$. It is well known that $X_m$ is a free $\Z/m\Z$-module of rank $2\dim(X)$ that is a Galois submodule of $X(\bar{K})$: we write
$$\bar{\rho}_{m,X}:\Gal(K) \to \Aut_{\Z/m\Z}(X_m)$$
for the corresponding structure homomorphism and
$$\tilde{G}_{m, X,K} \subset  \Aut_{\Z/m\Z}(X_m)$$
for its image.  A polarization $\lambda$  on $X$  that is defined over $K$ gives rise to the   Galois-equivariant alternating bilinear Riemann form
$$X_m\times X_m \to \mu_m$$
where $\mu_m$ is the cyclic group of all $m$th roots of unity in $\bar{K}$. Identifying (non-canonically)  $\mu_m$ with  $\Z/m\Z$, we may view the Riemann form  as an   alternating bilinear Riemann form
$$\bar{e}_{\lambda,m}: X_m \times X_m \to \Z/m\Z$$
such that
$$\bar{e}_{\lambda,m}(\sigma(x),\sigma(y))=\bar{\chi}_m(\sigma) \bar{e}_{\lambda,m}(x,y)$$
for all $x,y \in X_m$ and $\sigma \in \Gal(K)$ where
$$\bar{\chi}_{m}=\bar{\chi}_{m,K}: \Gal(K) \to (\Z/m\Z)^{*}$$
is the cyclotomic character that describes the Galois action on $m$th roots of unity.  (This form is nondegenerate if and only if $\deg(\lambda)$ and $m$ are relatively prime. In particular, if $\lambda$ is a principal polarization then $\bar{e}_{\lambda,m}$ is nondegenerate for all $m$.)
This implies that
$$\tilde{G}_{m, X,K}\subset
\Gp(X_m, \bar{e}_{\lambda,m})\subset     \Aut_{\Z/m\Z}(X_m).$$
Clearly, $\tilde{G}_{m, X,L}=\bar{\rho}_{m,X}(\Gal(K))$ is a subgroup of $\tilde{G}_{m, X,K}$ with index $\le [L:K]$.

If we choose a prime $\ell \ne \fchar(K)$, put $m=\ell^i$ and take the projective limit then we  get the Tate module $T_{\ell}(X)$ that is a free $\Z_{\ell}$-module of rank $2\dim(X)$ provided with the continuous Galois action ($\ell$-adic representation)
$$\rho_{\ell, X}:\Gal(K) \to \Aut_{\Z_{\ell}}(T_{\ell}(X))$$
and nondegenerate $\Z_{\ell}$-bilinear alternating Riemann form
$$e_{\lambda,\ell}: T_{\ell}(X) \times  T_{\ell}(X)  \to \Z_{\ell}$$
such that
$$e_{\lambda,\ell}(\sigma(x),\sigma(y))=\chi_{\ell}(\sigma) e_{\lambda}(x,y)$$
for all $x,y \in T_{\ell}(X)$ and $\sigma \in \Gal(K)$
where
$$\chi_{\ell}:\Gal(K) \to \Z_{\ell}^{*} \subset \Q_{\ell}^{*}$$
is the cyclotomic character that describes the Galois action on $\ell$-power roots of unity in $\bar{K}$.  (This form  is perfect if and only if $\deg(\lambda)$ is not divisible by $\ell$.)

 It follows that the image
$$G_{\ell,X,K}:=\rho_{\ell,X}(\Gal(K))\subset \Aut_{\Z_{\ell}}(T_{\ell}(X))$$
sits in the group $\Gp(T_{\ell}(X), e_{\lambda,\ell})$ of symplectic similitudes, i.e.,
$$G_{\ell,X,K}\subset \Gp(T_{\ell}(X), e_{\lambda,\ell}) \subset \Aut_{\Z_{\ell}}(T_{\ell}(X)).$$
Clearly,  $G_{\ell,X,L}:=\rho_{\ell,X}(\Gal(L))$ is a closed subgroup in $G_{\ell,X,K}$ with finite index $\le [L:K]$ and therefore is open in $G_{\ell,X,K}$.

In \cite[Th. 5.4 on p. 38]{ZarhinPLMS2}
 the author proved
the following statement.

\begin{thm}
\label{PLMSendo} Suppose that $\fchar(K)=0$ and $n=2g+2\ge 12$ is even. Assume also that
$f(x)=(x-t_1)(x-t_2) u(x)$ where
$$t_1, t_2 \in K, \ t_1 \ne t_2, \ u(x) \in K[x], \ \deg(u)=n-2=2g$$
and $\Gal(u)=\ST_{2g}$ or $\A_{2g}$.
 Then
 $\End(J(C_f))=\Z$. In particular, $J(C_f)$ is an absolutely simple abelian variety.
\end{thm}

The following statement follows easily from  \cite[Th. 8.3 on p. 49]{ZarhinPLMS2} applied to $t=t_1$ and $h(x)=(x-t_2)u(x)$ and an elementary substitution described in  \cite[Proof of Th. 5.4 on p. 38]{ZarhinPLMS2}.

 \begin{thm}
\label{PLMSimage} Suppose that $\K$ is a field that is finitely generated over $\Q$ and $n=2g+2\ge 12$ is even. Assume also that
$f(x)=(x-t_1)(x-t_2) u(x)$ with
$$t_1, t_2 \in K, \ t_1 \ne t_2, \ u(x) \in K[x], \ \deg(u)=n-2=2g$$
and $\Gal(u)=\ST_{2g}$ or $\A_{2g}$. Let $\lambda$ be the canonical principal polarization on the jacobian $J(C_f)$.
 Then the group $G_{\ell,J(C_f),K}$ is an open subgroup of finite index in the group $\Gp(T_{\ell}(J(C_f)), e_{\lambda,\ell})$ of symplectic similitudes.
\end{thm}

The aim of this note is, by imposing certain additional  {\sl arithmetic} conditions (inspired by \cite{HallBLMS}) on $f(x)$, to obtain the results about the groups
$\tilde{G}_{\ell,J(C_f),K}$ for almost all $\ell$ when $K$ is  a finitely generated field. In a sense, our approach is a combination of methods of \cite{ZarhinPLMS2} and \cite{HallBLMS}.
As a bonus, we were able to decrease lower bound for $g$ and cover the case  when $K$ has  prime characteristic.
Our main result is the following statement.

\begin{thm}
\label{mainENDO}
Let $g \ge 3$ be an integer. Let $K$ be a discrete valuation field, let $R\subset K$ be the discrete valuation ring  with maximal ideal $\m$ and residue field $k=R/\m$ of characteristic different from $2$. (In particular, $\fchar(K)\ne 2$.)  Let
$$u(x)=\sum_{i=0}^{2g} a_i x^i \in  K[x]$$
be a degree $2g$ polynomial that enjoys the following properties.

\begin{itemize}
\item[(i)]
The polynomial $u(x)$ is irreducible over $K$ and has no multiple roots, and its Galois group $\Gal(u)$ is either $\ST_{2g}$ or $\A_{2g}$.
\item[(ii)]
All the coefficients $a_i$ lie in $R$, i.e., $u(x)\in R[x]$.
\item[(iii)]
Neither the leading coefficient $a_{2g}$ nor the discriminant of $u(x)$ lie in $\m$. In other words $u(x)$ modulo $\m$ is a degree $2g$ polynomial over $k$ without multiple roots.
\end{itemize}

Suppose that $t_1$ and $t_2$ are two distinct elements of $R$ such that
$$t_1-t_2 \in \m, \ u(t_1) \not\in \m, \ u(t_2) \not\in \m .$$
Then

  $\End(J(C_f))=\Z$ where $f(x)=(x-t_1)(x-t_2) u(x)$. In particular, $J(C_f)$ is an absolutely simple abelian variety.

If, in addition, $K$ is a field  that is finitely generated over its prime subfield then:

\begin{itemize}
\item[(i)]
For all primes $\ell \ne \fchar(K)$  the group $G_{\ell,J(C_f),K}$ is an open subgroup of finite index in the group $\Gp(T_{\ell}(J(C_f)), e_{\lambda,\ell})$.
\item[(ii)]
If $L/K$ is a finite algebraic field extension then
for all but finitely many primes $\ell$ the group
$\tilde{G}_{\ell, ,J(C_f),L}$ contains
$\Sp(J(C_f)_{\ell}, \bar{e}_{\lambda,\ell})$ and  the group $G_{\ell,J(C_f),L}$ contains $\Sp(T_{\ell}(J(C_f)), e_{\lambda,\ell})$.
If, in addition, $\fchar(K)=0$ then for all but finitely many primes $\ell$
$$\tilde{G}_{\ell,J(C_f),L}=
\Gp(J(C_f)_{\ell}, \bar{e}_{\lambda,\ell}), \ G_{\ell,J(C_f),L}=\Gp(T_{\ell}(J(C_f)), e_{\lambda,\ell}).$$
\end{itemize}
\end{thm}

\begin{rem}
Suppose that $u(0)=a_0 \not\in \m$ (e.g., $a_0=\pm 1$). Then any pair $\{t_1, t_2\}$ of distinct elements of $\m$ satisfies the conditions of Theorem \ref{mainENDO} (for given $u(x)$).
\end{rem}

\begin{ex}
\label{mainEX}
Let $\Oc$ be a Dedekind ring with infinitely many maximal ideals and $K$ its field of fractions with $\fchar(K) \ne 2$. (E.g., $K$ is a number field  with ring of integers $\Oc$. Another
example: $\Oc$ is the ring of regular functions on an absolutely irreducible affine curve $\mathcal{C}$ over a field of characteristic different from $2$ and $K$ is the field of rational functions on $\mathcal{C}$.)
Let  $g>1$ be an integer,  and $u(x)=\sum_{i=0}^{2g} a_i x^i \in  \Oc[x]$  a degree $2g$ polynomial that is irreducible over $K$.  Pick any maximal ideal $\P$ in $\Oc$  such that the characteristic of the residue field  $\Oc/\P$ is different from
 $2$ and such that $a_0, a_n$ and the discriminant of $f(x)$ are $\P$-adic units. (This rules out only finitely many maximal ideals in $\Oc$.)
 Let us consider the discrete valuation ring $R$ that is the localization $\Oc_{\P}$ of $\Oc$ at $\P$. Then the residue field $k$ of $R$ coincides with $\Oc/\P$ and therefore has odd characteristic. Clearly,
$a_0, a_n$ and the discriminant of $f(x)$ are units in $R$.  Let $t_1, t_2$ be distinct elements of $\P$. Then they both lie in the maximal ideal of $R$.
 Now it's clear that if $g\ge 3$ then $\{K, R, u(x), t_1,t_2\}$ satisfy the conditions of Theorem \ref{mainENDO}.

For example, let $K=\Q, \Oc=\Z$ and $u(x)=x^{2g}-x-1$.
It is known \cite[Remark 2 at the bottom of p. 43]{SerreG} that $u(x)$ is irreducible over $\Q$ and its Galois group is $\ST_{2g}$.
In order to figure out for which prime $p$ the polynomial $u(x) \bmod p$ acquires multiple roots, we follow Serre's arguments (ibid).
 So, let us consider the polynomial $\bar{u}(x)=x^{2g}-x-1 \in \F_p[x]$ and assume that it has a multiple root say, $\alpha$. T
hen $\alpha$ is a also a root of the derivative $\bar{u}^{\prime}(x)=2g x^{2g-1}-1 \in \F_p[x]$.
It follows that $p$ does {\sl not} divide $2g$ and $\alpha \ne 0$.
  Clearly, $\alpha$ is a root of $2g \bar{u}(x)- x\bar{u}^{\prime}(x)=(1-2g)x-2g$.
This implies that $p$ does {\sl not} divide $2g-1$ and $\alpha =2g/(1-2g) \in \F_p$. This implies that $(2g)^{2g}/(1-2g)]^{2g-1}-1=0$ in $\F_p$,
  i.e.,  the integer $N(g)=(2g)^{2g}-(1-2g)^{2g-1}$ is divisible by $p$. In other words, the prime divisors of the discriminant of $u(x)$ are exactly the prime divisors of $N(g)$.  (Clearly, any prime divisor of $2g(2g-1)$ does not divide $N(g)$.) Now we  take any odd prime $p$ that does not divide $N(g)$ and pick any pair of {\sl distinct} integers $s_1, s_2$, and put $t_1=ps_1, t_2=ps_2$. Then $\{\Q, \Z_{(p)}, x^{2g}-x-1, t_1,t_2\}$  satisfy the conditions of Theorem \ref{mainENDO}. This implies that if we put $f(x)=(x^{2g}-x-1)(x-t_1)(x-t_2)$ then the jacobian $X=J(C_f)$ of the hyperelliptic curves $C_f:y^2=f(x)$  is an absolutely simple $g$-dimensional abelian variety over $K=\Q$ that enjoys the following properties.

$\End(X)=\Z$; for all primes $\ell$ the group $G_{\ell,X,K}$ is an open subgroup of finite index in $\Gp(T_{\ell}(X),e_{\lambda, \ell})$. In addition, if $L$ is a number field then for all but finitely many primes $\ell$
$$G_{\ell,X,L}=\Gp(T_{\ell}(X),e_{\lambda, \ell}), \  \tilde{G}_{\ell,X,L}=\Gp(X_{\ell},\bar{e}_{\lambda, \ell}).$$
\end{ex}

\begin{rem}
Earlier Chris Hall \cite{HallBLMS} proved  an analogue of Theorem \ref{mainENDO}:
in his result  $f(x)$ is required to be an  irreducible polynomial of degree $n\ge 5$  over a number field $K$ with coefficients
in the ring of integers of $K$ and Galois group $\ST_n$, and  such that modulo some odd prime it acquires exactly one multiple root and its multiplicity is $2$.
(His proof makes use of results of \cite{ZarhinMRL}.) It was proven by Emmanuel Kowalski (in an appendix to  \cite{HallBLMS}) that
 most of polynomials enjoy this property. It would be interesting to produce  explicit examples of such $f(x)$.
(E.g.,  arguments of \cite[p. 42, Remark 2]{SerreG} imply that $f(x)=x^n-x-1$ enjoys this property.)
However, Example \ref{mainEX} tells us how
 to produce a plenty of explicit examples of $f(x)$ that satisfy the conditions of Theorem \ref{mainENDO}.
\end{rem}

The next result tells us that distinct (unordered) pairs $(t_1,t_2)$ with given $u(x)$ (as in Theorem \ref{mainENDO}) lead to non-isomorphic (over $\bar{K}$) jacobians $J(C_f)$.

\begin{thm}
\label{distinctJAC}
Let $g \ge 2$ be a positive integer, $K$  a field of characteristic different from $2$, $u(x) \in K[x]$ an irreducible polynomial of degree $2g$ and without multiple roots. Assume that $\Gal(u)=\ST_{2g}$ or $\A_{2g}$. Let $r$ be an even positive integer, and let $B_1$ and $B_2$ be two distinct $r$-element subsets of $K$. Let us put
$$f_1(x)=u(x)\prod_{\alpha\in B_1}(x-\alpha) \in K[x], \ f_2(x)=u(x)\prod_{\alpha\in B_2}(x-\alpha) \in K[x].$$
Suppose that
$$\End(J(C_{f_1}))=\Z, \ \End(J(C_{f_2}))=\Z.$$
Then the jacobians $J(C_{f_1})$ and $J(C_{f_2})$ are not isomorphic over $\bar{K}$.
\end{thm}

The paper is organized as follows.
In Section \ref{group} we discuss the standard $(2g)$-dimensional permutational representation of the alternating group $\A_{2g}$ in characteristic $2$.
Section \ref{AV} deals with $g$-dimensional abelian varieties $X$ such that
 the absolute Galois group of the ground field acts on $X_2$ through its quotient isomorphic to $\A_{2g}$ and the $\A_{2g}$-module $X_2$ is isomorphic to the permutational one.
Examples of such $X$ are provided by certain hyperelliptic jacobians that are discussed in Section \ref{hyper2}; among them are jacobians that satisfy the conditions of Theorem \ref{mainENDO}.
 We prove Theorem \ref{mainENDO} in Section \ref{monodromy}.  In Section \ref{cycloCHAR0} we prove auxiliary results about Galois groups of cyclotomic extensions. In Section \ref{fractional} we prove Theorem \ref{distinctJAC}.
 Section \ref{conclude} contains (more or less straightforward) corollaries that tell us that the hyperelliptic jacobians involved (and their self-products) satisfy the Tate, Hodge and Mumford-Tate conjectures.

 {\bf Acknowledgements} I am grateful to Chris Hall for useful discussions, Gregorz Banaszak and Wojciech Gajda for stimulating questions, and Boris Kunyavskii for help with references. My special thanks go to the referee, whose comments helped to improve the exposition.

 This work was started during my stay at Max-Planck-Institut f\"ur Mathematik in Serptember of 2013 and finished
 during the academic year 2013/2014 when I was Erna and Jakob Michael Visiting Professor in the Department of Mathematics  at the Weizmann Institute of Science: the hospitality and support of both Institutes are gratefully acknowledged.

\section{Permutational representations of alternating groups}
\label{group}

\begin{sect}
\label{cover}
 Recall \cite{FT} that a surjective homomorphism of finite groups
$\pi:\GG_1\twoheadrightarrow \GG$ is called a {\sl minimal cover}
if no proper subgroup of $\GG_1$ maps onto $\GG$.
In particular, if $\GG$ is perfect and $\GG_1\twoheadrightarrow \GG$ is a
minimal cover then $\GG_1$ is also perfect.
In addition, if $r$ is a positive integer such that every subgroup in $\GG$ of index dividing $r$
coincides with $\GG$ then the same is true for $\GG_1$ \cite[Remark 3.4]{ZarhinMZ}. Namely,
every subgroup in $\GG_1$ of index dividing $r$ coincides with $\GG_1$.
\end{sect}

\begin{lem}
\label{An} Let $m\ge 5$ be an integer, $\A_m$ the corresponding alternating
group and $\GG_1 \twoheadrightarrow \A_m$
 a minimal cover.

 Then
the only subgroup of index $<m$ in $\GG_1$ is $\GG_1$ itself.
\end{lem}

\begin{proof}
This is Lemma 2.2(i) of \cite{ZarhinPLMS2}.

\end{proof}

\begin{sect}
Let $g\ge 3$ be an integer. Then $2g \ge 6$ and $\A_{2g}$ is a simple
nonabelian group.

Let $B$ be a $2g$-element set. We write $\Perm(B)$ for the group of all
permutations of $B$. The choice of ordering on $B$ establishes an isomorphism
between $\Perm(B)$ and the symmetric group $\mathrm{S}_{2g}$. We write
$\Alt(B)$ for the only subgroup of index $2$ in $\Perm(B)$. Every
isomorphism $\Perm(B)\cong\mathrm{S}_{2g}$ induces an isomorphism between
$\Alt(B)$ and the alternating group $\A_{2g}$. Let us consider the
$2g$-dimensional $\F_2$-vector space $\F_2^B$ of all $\F_2$-valued functions on
$B$ provided with the natural structure of faithful $\Perm(B)$-module. Notice
that the standard symmetric bilinear form
$$\F_2^B \times \F_2^B\to \F_2, \ (\phi,\psi)\mapsto \sum_{b\in
B}\phi(b)\psi(b)$$ is non-degenerate and $\Perm(B)$-invariant.

Since $\Alt(B)\subset \Perm(B)$, one may view $\F_2^B$ as a faithful
$\Alt(B)$-module.
\end{sect}

\begin{lem}
\label{AA8}
\begin{itemize}
\item[(i)] The centralizer $\End_{\Alt(B)}(\F_2^B)$ has $\F_2$-dimension $2$.

\item[(ii)] Every proper non-zero $\Alt(B)$-invariant subspace in $\F_2^B$ has
dimension $1$ or $2g-1$. In particular, $\F_2^B$ does not contain a proper
non-zero $\Alt(B)$-invariant even-dimensional subspace.
\end{itemize}
\end{lem}

\begin{proof} This is Lemma 2.5 of \cite{ZarhinPLMS2}
(Since $\Alt(B)$ is doubly transitive, (i) follows from \cite[Lemma
7.1]{Passman}.)

\end{proof}

\section{Abelian varieties}
\label{AV}

Let $F$ be a field, $\bar{F}$ its algebraic closure  and
$\Gal(F):=\Aut(\bar{F}/F)$ the absolute Galois group of $F$.

Recall that if $X$ is an abelian variety of positive dimension
 over $\bar{F}$ then we write $\End(X)$ for the ring of all its
$\bar{F}$-endomorphisms and $\End^0(X)$ for the corresponding $\Q$-algebra
$\End(X)\otimes\Q$. We write $\End_F(X)$ for the ring of all  $F$-endomorphisms
of $X$ and $\End_F^0(X)$ for the corresponding  $\Q$-algebra
$\End_F(X)\otimes\Q$ and $\CC$ for the center of $\End^0(X)$. Both $\End^0(X)$
and $\End_F^0(X)$  are semisimple finite-dimensional $\Q$-algebras.

 The absolute Galois  group
$\Gal(F)$ of $F$ acts on $\End(X)$ (and therefore on $\End^0(X)$)
by ring (resp. algebra) automorphisms and
$$\End_F(X)=\End(X)^{\Gal(F)}, \ \End_F^0(X)=\End^0(X)^{\Gal(F)},$$
since every endomorphism of $X$ is defined over a finite separable
extension of $F$.

\begin{thm}
\label{toricONE}
Let $X$ be an abelian variety of positive dimension over a field $K$ such that $\End^0(X)$ is a simple $\Q$-algebra, i.e., its center $\CC$ is a field.
 Suppose that $K$ a discrete valuation field  with discrete  valuation ring $R$ and residue field $k$.
Suppose that there exists a semiabelian group scheme $\X$ over $\Spec(R)$, whose generic fiber coincides with $X$ and the identity component
 $\X_k^{0}$  of the closed fiber $\X_k$ has toric dimension one, i.e., is a commutative algebraic $k$-group that is an extension of an abelian variety  by a one-dimensional algebraic torus.

Then $\End(X)=\Z$.
\end{thm}

\begin{proof}[Proof of Theorem \ref{toricONE}]
Extending $K$, we may and will assume that all endomorphisms of $X$ are defined over $k$.
 Removing from $\X$ all the irreducible components of $\X_k$ that do not pass through the identity element,
we may and will assume that $\X_k=\X_k^{0}$, i.e., the closed fiber of $\X$ is connected. It is known (\cite[Ch. IX, Cor. 1.4 on p. 130]{Raynaud}, \cite[Ch. 1, Sect. 2, Prop. 2.7, p. 9]{FaltingsChai} that every endomorphism of $X$ extends uniquely to to a certain endomorphism of the group scheme $\X/Spec(R)$. This gives us a  ring homomorphism
$$\End(X) \to \End(\X/\Spec(R))$$ that sends $1$ to $1$.
Composing it with the restriction homomorphism $\End(\X/\Spec(R)) \to \End(\X_k)$, we get a ring homomorphism $\End(X)\to \End(\X_k)$ that sends $1$ to $1$.

Let $\T$ be the one-dimensional torus in  $\X_k$. Clearly, $\End(\T)=\Z$. On the other hand,  every endomorphism of the algebraic $k$-group $\X_k$ leaves invariant $\T$, so we get the restriction ring  homomorphism $\End(\X_k) \to \End(\T)=\Z$ that sends $1$ to $1$. Taking the composition, we get the ring homomorphism
$$\End(X) \to \End(\T)=\Z$$
that sends $1$ to $1$. Extending the latter homomorphism by $\Q$-linearity, we get the homomorphism
$$\End^{0}(X) \to \Q$$
that sends $1$ to $1$. Since $\End^0(X)$ is a simple $\Q$-algebra, the latter homomorphism is an embedding and therefore $\End^0(X)=\Q$. This implies that $\End(X)=\Z$.
\end{proof}

\begin{cor}
\label{simpletoricONE}
Let $X$ be an absolutely simple abelian variety of positive dimension over a field $K$.  Suppose that $K$ a discrete valuation field  with discrete  valuation ring $R$ and residue field $k$. Suppose that there exists a semiabelian group scheme $\X$ over $\Spec(R)$, whose generic fiber coincides with $X$ and the identity componen $\X_k^{0}$  of the closed fiber $\X_k$ has toric dimension one, i.e.,  $\X_k^{0}$ is a commutative algebraic $k$-group that is an extension of  an abelian variety
by a one-dimensional algebraic torus.

Then $\End(X)=\Z$.
\end{cor}

\begin{proof}[Proof of Corollary \ref{simpletoricONE}]
The absolute simplicity of $X$ means that $\End^0(X)$ is a division algebra over $\Q$ and therefore is a simple $\Q$-algebra.
\end{proof}

Let $n$ be a positive integer that is not divisible by $\fchar(F)$. Recall that
 if $X$ is defined over $F$ then $X_n$ is a Galois
submodule in $X(\bar{F})$,   all points of $X_n$ are defined over a finite
separable extension of $F$ and we write $\bar{\rho}_{n,X,F}:\Gal(F)\to
\Aut_{\Z/n\Z}(X_n)$ for the corresponding homomorphism defining the structure
of the Galois module on $X_n$,
$$\tilde{G}_{n,X,F}\subset
\Aut_{\Z/n\Z}(X_{n})$$ for its image $\bar{\rho}_{n,X,F}(\Gal(F))$. We write
 $F(X_n)$ for the field of definition of all points of $X_n$.
Clearly, $F(X_n)$ is a finite Galois extension of $F$ with Galois
group $\Gal(F(X_n)/F)=\tilde{G}_{n,X,F}$. If $n=2$  then we get a
natural faithful linear representation
$$\tilde{G}_{2,X,F}\subset \Aut_{\F_{2}}(X_{2})$$
of $\tilde{G}_{2,X,F}$ in the $\F_{2}$-vector space $X_{2}$.

If $F_1/F$ is a finite algebraic extension then $F_1(X_n)$ coincides with the
compositum $F_1 F(X_n)$ of $F_1$ and $F(X_n)$

\begin{lem}
\label{goursat2}
 Let $F_1/F$ be a finite solvable Galois extension of fields.
If $\tilde{G}_{n,X,F}$ is a simple nonabelian group then $F_1$ and $F(X_n)$ are
linearly disjoint over $F$ and $\tilde{G}_{n,X,F_1}=\tilde{G}_{n,X,F}$.
\end{lem}

\begin{proof}
This is Lemma 3.2 of \cite{ZarhinPLMS2}.
\end{proof}

Now and until the end of this Section we assume  that $\fchar(F)\ne 2$. It is
known \cite{Silverberg} that all endomorphisms of $X$ are defined over
$F(X_4)$; this gives rise to the natural homomorphism
$$\kappa_{X,4}:\tilde{G}_{4,X,F} \to \Aut(\End^0(X))$$ and
$\End_F^0(X)$ coincides with the subalgebra
$\End^0(X)^{\tilde{G}_{4,X,F}}$ of $\tilde{G}_{4,X,F}$-invariants
\cite[Sect. 1]{ZarhinLuminy}.

The field inclusion $F(X_2)\subset F(X_4)$ induces a natural
surjection \cite[Sect. 1]{ZarhinLuminy}
$$\tau_{2,X}:\tilde{G}_{4,X,F}\twoheadrightarrow\tilde{G}_{2,X,F}.$$

\begin{defn} We say that $F$ is 2-{\sl balanced} with respect to
$X$ if $\tau_{2,X}$ is a minimal cover. (See \cite{ElkinZ}.)
\end{defn}

\begin{rem}
\label{overL}
 Clearly, there always exists a subgroup $H
\subset \tilde{G}_{4,X,F}$ such that the induced homomorphism
$H\to\tilde{G}_{2,X,F}$ is surjective and a minimal cover. Let us put
$L=F(X_4)^H$. Clearly,
$$F \subset L \subset F(X_4), \ L\bigcap F(X_2)=F$$
and $L$ is a maximal overfield of $F$ that enjoys these properties. It is also
clear that $H$ and $L$ can be chosen that
$$F \subset L \subset F(X_4), \ L\bigcap F(X_2)=F,$$
$$F(X_2)\subset L(X_2),\ L(X_4)=F(X_4), \
\tilde{G}_{2,X,L}=\tilde{G}_{2,X,F}$$
 and $L$ is $2$-{\sl balanced}   with respect to
$X$ (\cite[Remark 2.3]{ElkinZ}; see also \cite{ElkinZ2}).
\end{rem}

We will need the following  result from our previous work.

\begin{lem}
\label{Ksimple} Assume that $X_2$ does not contain  a proper nonzero
$\tilde{G}_{2,X,F}$-invariant even-dimensional subspace and the centralizer
$\End_{\tilde{G}_{2,X,F}}(X_2)$ has $\F_2$-dimension $2$.

Then $X$ is $F$-simple and $\End_F^0(X)$ is either $\Q$ or a
quadratic field.
\end{lem}

\begin{proof} This is Lemma 3.4  of \cite{ZarhinMA}.
\end{proof}

\begin{thm}
\label{mainAV} Let $g\ge 3$ be an integer and $B$ a $2g$-element set. Let $X$
be a $g$-dimensional abelian variety over $F$. Suppose that there exists a
group isomorphism $\tilde{G}_{2,X,F}\cong \Alt(B)$ such that the
$\Alt(B)$-module $X_2$ is isomorphic to $\F_2^B$.

Then the center $\CC$ of $\End^0(X)$ is a field, i.e., $\End^0(X)$ is a finite-dimensional simple $\Q$-algebra.
\end{thm}

 \begin{proof}[Proof of Theorem \ref{mainAV}]
 By Remark \ref{overL}, we may and will assume that $F$ is 2-{\sl balanced} with respect to
 $X$, i.e., $\tau_{2,X}:\tilde{G}_{4,X,F}\twoheadrightarrow
\tilde{G}_{2,X,F}=\A_{2g}$ is a minimal cover. In particular,
$\tilde{G}_{4,X,F}$ is perfect, since $\A_{2g}$ is perfect. Since $\A_{2g}$
does not contain a  subgroup of index $<2g$
 different from $\A_{2g}$,
 it
follows from Lemma \ref{An}(i) that $\tilde{G}_{4,X,F}$  does not contain a
proper subgroup of index $<2g$
 different from $\tilde{G}_{4,X,F}$.
Now Lemmas
\ref{Ksimple} and \ref{AA8} imply that
$\End_F^0(X)$ is either $\Q$ or a quadratic field.

Recall that $\CC$ is the center of $\End^0(X)$.
Suppose that $\CC$ is {\sl not} a field. Then it is a direct sum
$$\CC=\oplus_{i=1}^r \CC_i$$
of number fields $\CC_1, \dots , \CC_r$ with $1<r\le \dim(X)=g$. Clearly, the
center $\CC$ is a $\tilde{G}_{4,X,F}$-invariant subalgebra of $\End^0(X)$; it
is also clear that $\tilde{G}_{4,X,F}$ permutes the summands $\CC_i$'s. Since
$\tilde{G}_{4,X,F}$ does not contain proper subgroups of index $\le g$, each
$\CC_i$ is $\tilde{G}_{4,X,F}$-invariant. This implies that the $r$-dimensional
$\Q$-subalgebra
$$\oplus_{i=1}^r \Q \subset \oplus_{i=1}^r \CC_i$$
consists of $\tilde{G}_{4,X,F}$-invariants and therefore lies in
 $\End_F^0(X)$. It follows that $\End_F^0(X)$ has zero divisors,
 which is not the case. The obtained contradiction proves that $\CC$
 is a field.
\end{proof}

\begin{cor}
\label{mainAVcor} Let $g\ge 3$ be an integer and $B$ a $2g$-element set. Let $X$
be a $g$-dimensional abelian variety over $F$. Suppose that there exists a
group isomorphism $\tilde{G}_{2,X,F}\cong \Alt(B)$ such that the
$\Alt(B)$-module $X_2$ is isomorphic to $\F_2^B$. Assume additionally that there exists a finite algebraic  field extension $E/F$ such that
$E$ is a discrete valuation field with discrete valuation ring $R$ and residue field $k$ such that the N\'eron model of $X$
over $\Spec(R)$ is a semiabelian group scheme, whose closed
fiber has toric dimension $1$.

Then $\End(X)=\Z$.
\end{cor}

\begin{proof}
 The result follows readily from Theorem \ref{mainAV} combined with Theorem \ref{toricONE}.
\end{proof}

\section{Abelian varieties with semistable reduction and toric dimension one}
This section is a variation on a theme of \cite{HallBLMS} (see also \cite{Gajda}).

Let $X$ be an abelian variety of positive dimension over a field $K$ with polarization $\lambda$ and let $\ell$ be a prime different from $\fchar(K)$. Let us consider the  $2\dim(X)$-dimensional $\Q_{\ell}$-vector space
$$V_{\ell}(X)=T_{\ell}(X)\otimes_{\Z_{\ell}}\Q_{\ell}.$$
One may view $T_{\ell}(X)$ as a $
\Z_{\ell}$-lattice of maximal rank; the Galois action on $T_{\ell}(X)$ extends by $\Q_{\ell}$-linearity to $V_{\ell}(X)$ and we may view $\rho_{\ell,X}$ as the $\ell$-adic representation
$$\rho_{\ell,X}:\Gal(K) \to \Aut_{\Z_{\ell}}(T_{\ell}(X)) \subset \Aut_{\Q_{\ell}}(V_{\ell}(X)).$$
and its image $G_{\ell,X,K}$ as a compact $\ell$-adic subgroup in $\Aut_{\Q_{\ell}}(V_{\ell}(X))$ \cite{SerreAbelian}. We write
$$\g_{\ell,X}\subset \End_{\Q_{\ell}}(V_{\ell}(X))$$
for the Lie algebra of $G_{\ell,X,K}$: it is a $\Q_{\ell}$-linear Lie subalgebra of $\End_{\Q_{\ell}}(V_{\ell}(X)$ that would not change if one replaces $K$ by its finite algebraic extension.
On the other hand, extending $e_{\lambda,\ell}$ by $\Q_{\ell}$-linearity to $V_{\ell}(X)$ from $T_{\ell}(X)$, we obtain the nondegenerate alternating $\Q_{\ell}$-bilinear form
$$V_{\ell}(X) \times V_{\ell}(X) \to \Q_{\ell},$$
which we continue to denote by $e_{\ell,\lambda}$.
We have
$$G_{\ell,X,K}\subset \Gp(T_{\ell}(X),e_{\ell,\lambda})\subset \Gp(V_{\ell}(X),e_{\ell,\lambda})\subset \Aut_{\Q_{\ell}}(V_{\ell}(X)).$$
It is well known that the Lie algebra $\gp(V_{\ell}(X),e_{\ell,\lambda})$ of the $\ell$-adic Lie group $\Gp(V_{\ell}(X),e_{\ell,\lambda})$ coincides with the direct sum
$\Q_{\ell}\I\oplus \sp(V_{\ell}(X),e_{\ell,\lambda})$ where $\I: V_{\ell}(X) \to V_{\ell}(X)$ is the identity map and $\sp(V_{\ell}(X),e_{\ell,\lambda})$ is the Lie algebra of the $\ell$-adic symplectic Lie group
$\Sp(V_{\ell}(X),e_{\ell,\lambda})$. We have
$$\g_{\ell,X}\subset \Q_{\ell}\I\oplus \sp(V_{\ell}(X),e_{\ell,\lambda})\subset \End_{\Q_{\ell}}(V_{\ell}(X)).$$
Notice that the open compact subgroup $\Gp(T_{\ell}(X),e_{\ell,\lambda})$ of $\Gp(V_{\ell}(X),e_{\ell,\lambda})$  has the same Lie algebra
$\Q_{\ell}\I\oplus \sp(V_{\ell}(X),e_{\ell,\lambda})$ as $\Gp(V_{\ell}(X),e_{\ell,\lambda})$.

Now assume that $K$ is finitely generated over its prime subfield and $\End(X)=\Z$.
According to results of \cite{ZarhinMZ1, F2,MB} (where the Tate conjecture for homomorphisms of abelian varieties and semisimplicity of Tate modules were proven), for every finite separable algebraic extension
$K_1$ of $K$ the $\Gal(K_1)$-module $V_{\ell}(X)$ is absolutely simple. We claim
 that
for every open subgroup $G_1$ of finite index in $G_{\ell,X,K}$ the $G_1$-module $V_{\ell}(X)$ is absolutely simple.
Indeed, the preimage $\rho_{\ell,X}^{-1}(G_1)$ is an open subgroup of finite index in $\Gal(K)$ and therefore coincides with $\Gal(K_1)$ for a certain  finite separable algebraic extension $K_1/K$; in addition,
$\rho_{\ell,X}(\Gal(K_1))=G_1$. It follows that
$$\Q_{\ell}=\End_{\Gal(K_1)}(V_{\ell}(X))=\End_{G_1}(V_{\ell}(X))$$
and we are done if we know that the $G_1$-module $V_{\ell}(X)$ is semisimple. However, if $G_1$ is normal in $G_{\ell,X,K}$ then the (semi)simplicity of the $G_{\ell,X,K}$-module $V_{\ell}(X)$
 implies the semisimplicity of the $G_1$-module $V_{\ell}(X)$, thanks to a theorem of Clifford \cite[Sect, 49, Th. (49.2)]{CR}. In order to do the general case of not necessarily normal $G_1$, notice that
every $G_1$ contains an open subgroup $G_2$ that is a normal (open) subgroup of finite index in $G_{\ell,X,K}$ that is the kernel of the natural continuous
 homomorphism from $G$ to the group of permutations
of the finite set $G_{\ell,X,K}/G_1$. We get that $V_{\ell}(X)$ is an absolutely simple $G_2$-module. Since $G_1$ contains $G_2$, it follows that the $G_1$-module $V_{\ell}(X)$ is also absolutely simple.

Applying Lemma 7.1 of \cite{ZarhinPLMS2} to $V=V_{\ell}(X), G=G_{\ell,X,K}, e=e_{\ell,\lambda}$,  we conclude that there exists a semisimple $\Q_{\ell}$-Lie algebra
$$\g^{\ss}=\g_{\ell}^{\ss}\subset \sp(V_{\ell}(X),e_{\ell,\lambda})\subset\End_{\Q_{\ell}}(V_{\ell}(X))$$
such that either $\g_{\ell,X}=\g^{\ss}$ or $\g_{\ell,X}=\Q_{\ell}\I\oplus\g^{\ss}$.
In addition.
$$\g^{\ss}=\g_{\ell}^{\ss}\subset \sp(V_{\ell}(X), e_{\ell,\lambda}).$$
Since $\End(X)=\Z$, it follows from \cite[Cor. 1.3.1]{ZarhinMZ3} (see also \cite{Bogomolov1,Bogomolov2}) that the center of  $\g_{\ell,X}$ coincides with $\Q_{\ell}\I$ and therefore
$$\g_{\ell,X}=\Q_{\ell}\I\oplus\g^{\ss}=\Q_{\ell}\I\oplus\g_{\ell}^{\ss}\subset \Q_{\ell}\I\oplus \sp(V_{\ell}(X),e_{\ell,\lambda}).$$
Clearly, the semisimple  linear $\Q_{\ell}$-Lie algebra
$$\g_{\ell}^{\ss}\subset\End_{\Q_{\ell}}(V_{\ell}(X))$$ is absolutely irreducible.

\begin{rem}
 \label{LieSumplectic}
Assume that $\g_{\ell}^{\ss}=\sp(V_{\ell}(X), e_{\ell,\lambda})$.
Then the Lie algebra $\Q_{\ell}\I\oplus\g_{\ell}^{\ss}$ of $G_{\ell,X.K}$ coincides with the Lie algebra
$\Q_{\ell}\I\oplus \sp(V_{\ell}(X),e_{\ell,\lambda})$ of compact $\Gp(T_{\ell}(X),e_{\ell,\lambda})$ and
therefore  $G_{\ell,X.K}$ is an open subgroup of finite index in  $\Gp(T_{\ell}(X),e_{\ell,\lambda})$.
\end{rem}

\begin{rem}
 \label{rankONE}
Suppose that the absolutely irreducible linear Lie algebra
 $$\g_{\ell,X}\subset \End_{\Q_{\ell}}(V_{\ell}(X))$$
contains a linear operator $V_{\ell}(X) \to V_{\ell}(X)$ of rank one. Let us look at
the classification (in characteristic zero) of absolutely irreducible linear Lie algebras with operator of rank one \cite{GQS} (see also \cite[Ch. 8, sect. 13, ex. 15]{Bourbaki}).
The list  consists of $\End_{\Q_{\ell}}(V_{\ell}(X))$, the Lie algebra $\sL(V_{\ell}(X))$ of all  operators with zero trace, $\sp(V_{\ell}(X))$ and $\Q_{\ell}\I\oplus \sp(V_{\ell}(X))$
where $\sp(V_{\ell}(X))$ is the Lie algebra of the symplectic group of a certain nondegenerate alternating bilinear form on $V_{\ell}(X)$. Since
$$\Q_{\ell}\I \subset \g_{\ell,X} \subset \Q_{\ell}\I\oplus \sp(V_{\ell}(X),e_{\ell,\lambda}),$$
we conclude that $\g_{\ell,X}$ coincides with $\Q_{\ell}\I\oplus \sp(V_{\ell}(X),e_{\ell,\lambda})$.
By Remark \ref{LieSumplectic}, $G_{\ell,X,K}$ is an open subgroup of finite index in  $\Gp(T_{\ell}(X),e_{\ell,\lambda})$.
\end{rem}

\begin{thm}
\label{tateToricOne}
Suppose that $K$ is finitely generated over its prime subfield and $\End(X)=\Z$. Assume additionally that there exists a finite algebraic  field extension $E/K$ such that
$E$ is a discrete valuation field with discrete valuation ring $R$ and residue field $k$ such that the N\'eron model $\X$ of $X$
over $\Spec(R)$ is a semiabelian group scheme, whose closed
fiber has toric dimension $1$. Suppose that $\fchar(k) \ne \ell$.
Then $G_{\ell,X,K}$ is an open subgroup of finite index in  $\Gp(T_{\ell}(X),e_{\ell,\lambda})$.
\end{thm}

\begin{proof}
 Replacing $K$  by $E$, we may and will assume that $E=K$. So, $K$ is the discrete valuation field with discrete valuation ring $\Oc$ and residue field $k$, the N\'eron model $\X$ of $X$ over $\Oc$ is a semiabelian scheme (with generic fiber $X$) such that the identity component of its closed  fiber
(over $k$)
 is an extension of a   $(\dim(X)-1)$-dimensional abelian variety by a one-dimensional torus.

 Let us choose a {\sl henselization} $\Oc^{h}\subset \bar{K}$ of $\Oc$ \cite[Sect. 2.3]{Neron}; it is a henselian discrete valuation ring containing $\Oc$ that has the same residue field $k$, and any uniformizer of $\Oc$ is also an uniformizer of $\Oc^h$.
 The field $K^h$ of fractions of  $\Oc^h$ is a discrete valuation field containing $K$. Since
$$K \subset K^h \subset \bar{K},$$
we may view $\Gal(K^h)$ as a (closed) subgroup of $\Gal(K)$.
  Let $\II \subset \Gal(K^h)$ be the corresponding inertia (sub)group \cite[Sect. 2.3, Prop. 11]{Neron}. We have
$$\II \subset \Gal(K^h)\subset \Gal(K).$$

It is known  \cite[Sect. 7.2,  Th. 1 and Cor. 2]{Neron} that the N\'eron model $\X^{h}$ of $X$ over $\Oc^h$ is canonically isomorphic to $\X \otimes_{\Oc}\Oc^{h}$. In particular, $X$ has semistable reduction over $K^h$ and  the identity component
 of its closed  fiber ${\X^{h}}_k$ is a commutative algebraic group
over $k$ that
 is an extension of a   $(\dim(X)-1)$-dimensional abelian variety by a one-dimensional torus; we denote this torus by $\T_0$. One may identify the $\ell$-adic  Tate module $T_{\ell}(\T_0)$ of $\T_0$ with a certain rank $\dim(\T_0)$ free $\Z_{\ell}$-submodule   $W$ of $T_{\ell}(X)$ that is called the {\sl toric part} of $T_{\ell}(X)$ \cite[Sect. 2.3]{GrothendieckN}.  (In our case $W$ has rank $1$.)

 Let  $T_{\ell}(X)^{\II}$ be the $\Z_{\ell}$-submodule of $\II$-invariants in $\T_{\ell}(X)$.
By   Grothendieck's criterion of semistable reduction \cite[Prop. 3.5(iii) on p. 350]{GrothendieckN}, the orthogonal complement of $T_{\ell}(X)^{\II}$ in  $T_{\ell}(X)$ with respect to $e_{\ell,\lambda}$ coincides with $W$. Since $e_{\ell,\lambda}$ is nondegenerate, the rank arguments imply that $T_{\ell}(X)^{\II}$ is a free $\Z_{\ell}$-module of rank $2\dim(X)-1$.  It follows easily that the $\Q_{\ell}$-vector subspace $V_{\ell}(X)^{\II}$ of $\II$-invariants has codimension $1$ in $V_{\ell}(X)$.
It follows that there exists
$$\sigma \in \II \subset \Gal(K^h)\subset \Gal(K)$$
 such that the subspace of $\sigma$-invariants in $V_{\ell}(X)$ has codimension $1$. This implies that the linear operator
$$u:=\rho_{\ell,X}(\sigma)-\I: V_{\ell}(X) \to V_{\ell}(X)$$
has rank one. The other part of the same criterion of Grothendieck \cite[Prop. 3.5(iv)]{GrothendieckN} implies that $\rho_{\ell,X}$ is an unipotent linear operator in $V_{\ell}(X)$; more precisely,
$$[\rho_{\ell,X}(\sigma)-\I]^2=0 \in \End_{\Q_{\ell}}(V_{\ell}(X)),$$
since the reduction is semistable. Then the $\ell$-adic logarithm
$\log(\rho_{\ell,X}(\sigma))$ of $\rho_{\ell,X}(\sigma)\in G_{\ell,X,K}$ equals $\rho_{\ell,X}(\sigma)-\I$ and therefore
coincides with $u$. Since $\log(\rho_{\ell,X}(\sigma))$ lies in the Lie algebra $\g_{\ell,X}$ of $G_{\ell,X,K}$, we conclude that $u$
is the desired operator of rank one in $\g_{\ell,X}$. Now Remark \ref{rankONE} implies that
$$\g_{\ell,X}=\Q_{\ell}\I\oplus \sp(V_{\ell}(X),e_{\ell,\lambda})$$
and $G_{\ell,X.K}$ is an open subgroup of finite index in  $\Gp(T_{\ell}(X),e_{\ell,\lambda})$.
\end{proof}

\begin{thm}
 \label{hallGajda}
We keep the notation and assumptions of Theorem \ref{tateToricOne}. Then for all but finitely many primes $\ell$ the group
$\tilde{G}_{\ell,X,K}$ contains $\Sp(X_{\ell},\bar{e}_{\lambda, \ell})$.
\end{thm}

\begin{proof}
 This is a result of \cite{HallBLMS} when $K$ is a global field. The general case was done in \cite{Gajda}. The proof   makes use of  the classification of  irreducible linear groups (over finite fields) generated by transvections \cite{ZS}.
\end{proof}

Let $K$ be a field that its finitely generated over its prime subfield. For each prime $\ell \ne \fchar(K)$ and positive integer $i$ we write
$K(\mu_{\ell^j})$ for the subfield of $\bar{K}$ obtained by adjoining to $K$ all $\ell^j$th roots of unity. It is well known that $K(\mu_{\ell^j})/K$ is an abelian
 field extension of degree dividing $(\ell-1)\ell^{j-1}$ and the cyclotomic character $\bar{\chi}_{\ell^j}$ factors through the embedding
$$\Gal(K(\mu_{\ell^j})/K) \hookrightarrow (\Z/\ell^j\Z)^{*}.$$

We will use the following elementary  statement that is well known but I did not find a suitable reference. (It  will be proven in Section \ref{cycloCHAR0}).

\begin{thm}
\label{cyclotomy}
 Let $K$ be a field that its finitely generated over its prime subfield.  Then for all but finitely many primes $\ell$  all the group embeddings
$\Gal(K(\mu_{\ell^j})/K) \hookrightarrow (\Z/\ell^j\Z)^{*}$ are isomorphisms.
\end{thm}

\begin{cor}[Corollary to Theorem \ref{tateToricOne}]
 \label{corToricOne}
We keep the notation and assumptions of Theorem \ref{tateToricOne}. Then for all but finitely many primes $\ell$ the group
$G_{\ell,X,K}$ contains $\Sp(T_{\ell}(X),e_{\ell,\lambda})$. If, in addition, $\fchar(K)=0$ then
for all but finitely many primes $\ell$ the group
$G_{\ell,X,K}=\Gp(T_{\ell}(X),e_{\ell,\lambda})$.
\end{cor}

\begin{proof}[Proof of Corollary \ref{corToricOne}]
 Let us assume that a prime $\ell \ge 5$ and $\deg(\lambda)$ is not divisible by $\ell$.
In particular, $\bar{e}_{\lambda,\ell}$ is nondegenerate and the finite group $\Sp(X_{\ell},\bar{e}_{\lambda, \ell})$ is perfect, i.e., coincides with its own derived
subgroup $[\Sp(X_{\ell},\bar{e}_{\lambda, \ell}),\ \Sp(X_{\ell},\bar{e}_{\lambda, \ell})]$.  Using Theorem \ref{hallGajda}, we may and will assume (after removing finitely many primes)
that $\tilde{G}_{\ell,X,K}$ contains $\Sp(X_{\ell},\bar{e}_{\lambda, \ell})$.

Recall that $G_{\ell,X.K}$ is an open subgroup of finite index in  $\Gp(T_{\ell}(X),e_{\ell,\lambda})$ and therefore is a {\sl closed} subgroup in  $\Gp(T_{\ell}(X),e_{\ell,\lambda})$.
Following Serre \cite{SerreV}, let us consider the closure $G$ of the derived subgroup $[G_{\ell,X,K}, \ G_{\ell,X,K}]$ of $G_{\ell,X,K}$ in
$\Gp(T_{\ell}(X),e_{\ell,\lambda})$. Since $G_{\ell,X,K}$ is closed  in  $\Gp(T_{\ell}(X),e_{\ell,\lambda})$, the group $G$ is a subgroup of $G_{\ell,X,K}$. Clearly, $G$ is also a closed subgroup of $\Sp(T_{\ell}(X),e_{\ell,\lambda})$ that maps surjectively on
$$[\Sp(X_{\ell},\bar{e}_{\lambda, \ell}),\ \Sp(X_{\ell},\bar{e}_{\lambda, \ell})]=\Sp(X_{\ell},\bar{e}_{\lambda, \ell}).$$
It follows from a theorem of Serre \cite{SerreV} (see also \cite[Th. 1.3]{Vasiu}) that $G=\Sp(T_{\ell}(X),e_{\ell,\lambda})$. We conclude that $\Sp(T_{\ell}(X),e_{\ell,\lambda}) \subset G_{\ell,X,K}$. This proves the first assertion.

Now, assume additionally that $\fchar(K)=0$.  It follows from Theorem \ref{cyclotomy} that for all but finitely many primes $\ell$ the cyclotomic character
$\chi_{\ell}: \Gal(K) \to \Z_{\ell}^{*}$ is surjective.  This implies that the homomorphism
$$G_{\ell,X,K}\to\Gp(T_{\ell}(X),e_{\ell,\lambda})/\Sp(T_{\ell}(X),e_{\ell,\lambda})=\Z_{\ell}^{*}$$
is also surjective for all but finitely many primes $\ell$. In order to finish the proof, one has only to recall that we just proved that
$\Sp(T_{\ell}(X),e_{\ell,\lambda}) \subset G_{\ell,X,K}$ for all but finitely many primes $\ell$.

\end{proof}

\begin{rem}
It follows from Theorem \ref{unboundK} below that when $\fchar(K)=p>0$ then the index  of the image $\bar{\chi}_{\ell}(\Gal(K))$  in $(\Z/\ell\Z)^{*}$ is an unbounded function in $\ell$.  It follows that the function that assigns to a prime $\ell \ne p$ the index of $\tilde{G}_{\ell,X,K}$ in $\Gp(X_{\ell},\bar{e}_{\lambda, \ell})$ is also unbounded. This, in turn,  implies the unboundness of the function that that assigns to a prime $\ell \ne p$ the index of $G_{\ell,X,K}$ in $\Gp(T_{\ell}(X),e_{\lambda, \ell})$.
\end{rem}

\begin{rem}
It was stated without a proof in \cite[Remark on p. 707]{HallBLMS} that if $K$ is a global field of characteristic $p>0$
then $\tilde{G}_{\ell,X,K}$ does {\sl not} coincide with $\Gp(X_{\ell},\bar{e}_{\lambda, \ell})$ for {\sl infinitely} many primes $\ell$.
\end{rem}

\begin{rems}
\label{TateHodge}

\begin{itemize}
\item[(i)]
Recall (see the proof of Theorem \ref{tateToricOne}) that if a prime $\ell \ne \fchar(k)$ then
$$\g_{\ell,X}=\Q_{\ell}\I\oplus \sp(V_{\ell}(X),e_{\ell,\lambda}).$$
I claim that this equality (and therefore the conclusion of Theorem \ref{tateToricOne})  hold for all $\ell \ne \fchar(K)$. Clearly, the only  remaining case is
$$\fchar(K)=0, \fchar(k)=p>0, \ell=p.$$
In order to do that, let us choose a {\sl prime} $q \ne p$. We know that
$$\g_{q,X}=\Q_{q}\I\oplus \sp(V_{q}(X),e_{q,\lambda}).$$
Recall that for {\sl all} primes $\ell$
$$\g_{\ell,X}=\Q_{\ell}\I\oplus\g_{\ell}^{\ss}\subset \Q_{\ell}\I\oplus \sp(V_{\ell}(X),e_{\ell,\lambda}),$$
where $\g_{\ell}^{\ss}$ is an absolutely irreducible semisimple $\Q_{\ell}$-Lie algebra.
Now the same arguments as in  \cite[Lemma 8.2 and its Proof on pp. 426--427]{ZarhinMMJ}  prove that
$$\g_{p,X}=\Q_{p}\I\oplus \sp(V_{p}(X),e_{p,\lambda})$$
provided we replace  all $2$ by $q$ and all  $\ell$ by $p$.

\item[(ii)]

The same arguments from invariant theory \cite{Howe} as in \cite[Sect. 9]{ZarhinPLMS2} prove that for every finite algebraic field extension $K^{\prime}/K$  and each
self-product $X^m$ of $X$ every $\ell$-adic Tate class on $X^m$ can be presented as a linear combination of products of divisor classes on $X^m$. In particular, the Tate conjecture
 holds true for all $X^m$ in all codimensions. (In codimension one the Tate conjecture \cite{Tate} for abelian varieties was proven by Tate himself over finite fields \cite{TateInv},
by the author \cite{ZarhinMZ1} in characteristic
$>2$, by Faltings \cite{F1,F2} in characteristic $0$, and by S. Mori \cite{MB} in characteristic $2$ respectively.)

Assume additionally that $\fchar(K)=0$ and therefore $K$ is finitely generated over $\Q$, and fix an embedding $\bar{K}\subset \C$. Then the same arguments as in \cite[Sect.  10]{ZarhinMMJ} and
 \cite[Sect. 10]{ZarhinPLMS2}  (based on a theorem of Pijatetskij-Shapiro, Deligne and Borovoi \cite{Deligne,SerreKyoto}) prove that for each
self-product $X^m$ of $X$ every Hodge class on $X^m$ can be presented as a linear combination of products of divisor classes on $X^m$. In particular, the Hodge conjecture
 holds true for all $X^m$ in all codimensions. In addition, the Mumford-Tate conjecture holds true for $X$. (See also \cite{ZarhinJussieu}.)
\end{itemize}
\end{rems}

\section{Points of order 2}
\label{hyper2}
\begin{sect}
\label{heart} Let $K$ be a field of characteristic different from $2$, let
$f(x)\in K[x]$ be a polynomial of {\sl odd} degree $n\ge 5$ and without
multiple roots. Let $C_f$ be the hyperelliptic curve $y^2=f(x)$ and $J(C_f)$
the jacobian of $C_f$.
  The Galois module $J(C_f)_2$ of points of order $2$ admits the following
  description.

  Let $\F_2^{\RR_f}$ be the $n$-dimensional $\F_2$-vector space of functions $\varphi: \RR_f \to \F_2$
   provided with the
   natural structure of $\Gal(f)\subset \Perm(\RR_f)$-module. The canonical surjection
    $$\Gal(K)\twoheadrightarrow \Gal(K(\RR_f)/K)=\Gal(f)$$
    provides $\F_2^{\RR_f}$ with the structure of $\Gal(K)$-module. Let us
    consider the hyperplane
    $$(\F_2^{\RR_f})^0:=\{\varphi:\RR_f \to \F_2\mid \sum_{\alpha\in \RR_f}\varphi(\alpha)=0\}\subset \F_2^{\RR_f}.$$
    Clearly, $(\F_2^{\RR_f})^0$ is a Galois submodule in $\F_2^{\RR_f}$.

    It is well known (see, for instance, \cite{ZarhinTexel}) that if $n$ is odd
    then the Galois modules $J(C_f)_2$ and
$(\F_2^{\RR_f})^0$ are isomorphic. It follows that if $X=J(C_f)$ then
$\tilde{G}_{2,X,K}=\Gal(f)$ and $K(J(C_f)_2)=K(\RR_f)$.
\end{sect}

\begin{lem}
\label{order2} Suppose that $n=\deg(f)$ is odd and $f(x)=(x-t)h(x)$ with $t\in
K$ and $h(x)\in K[x]$. Then $\tilde{G}_{2,J(C_f),K}\cong \Gal(h)$ and the
Galois modules $J(C_f)_2$ and $\F_2^{\RR_h}$ are isomorphic.
\end{lem}

\begin{proof}
This is Lemma 5.1 of \cite{ZarhinPLMS2}.
\end{proof}

\begin{cor}
\label{helpcor} Suppose that $n=\deg(f)=2g+1$ is odd  and
$f(x)=(x-t)h(x)$ with $t\in K$ and $h(x)\in K[x]$. Assume also that
$\Gal(h)=\Alt(\RR_h)\cong \A_{2g}$.

Assume additionally that there exists a finite algebraic  field extension $E/K$ such that
$E$ is a discrete valuation field with discrete valuation ring $R$ and residue field $k$ such that the N\'eron model of $J(C_f)$
over $\Spec(R)$ is a semiabelian group scheme, whose closed
fiber has toric dimension $1$.

Then $\End(J(C_f))=\Z$.
\end{cor}

\begin{proof}[Proof of Corollary \ref{helpcor}]
Let us put $K=F$, $X=J(C_f)$ and $B=\RR_h$.
 Then assertion  is an immediate corollary of Lemma \ref{order2}
and Corollary
\ref{mainAVcor}.
\end{proof}

\begin{thm}
\label{twodim}
Suppose that $n=2g+2=\deg(f)\ge 8$ is even and
$f(x)=(x-t_1)(x-t_2)u(x)$ with
$$t_1, t_2 \in K, \ t_1 \ne t_2, \ u(x)\in K[x], \ \deg(u)=n-2.$$
Suppose that $\Gal(u)=\ST_{2g}$ or $\A_{2g}$.
Assume additionally that there exists a finite algebraic  field extension $E/K$ such that
$E$ is a discrete valuation field with discrete valuation ring $R$ and residue field $k$ such that the N\'eron model of $J(C_f)$
over $\Spec(R)$ is a semiabelian group scheme, whose closed
fiber has toric dimension $1$.
Then $\End(J(C_f))=\Z$.
\end{thm}

\begin{proof}
Replacing if necessary, $K$ by its suitable quadratic extension, we may and will assume that $\Gal(u)=\A_{2g}$.
Let us put $h(x)=(x-t_2)u(x)$. We have $f(x)=(x-t_1)h(x)$. Let us consider the degree $(n-1)$ polynomials

$$h_1(x)=h(x+t_1)=(x+t_1-t_2)u(x+t_1), \ h_2(x)=x^{n-1} h_1(1/x) \in K[x].$$
 We have
$$\RR_{h_1}=\{\alpha-t_1\mid \alpha \in \RR_{h}\}=
\{\alpha-t_1+t_2\mid \alpha \in \RR_{u}\} \bigcup \left\{t_2-t_1\right\},$$
$$\RR_{h_2}=\left\{\frac{1}{\alpha-t_1}\mid \alpha \in \RR_u\right\}\bigcup
\left\{\frac{1}{t_2-t_1}\right\}.$$ This implies that
$$K(\RR_{h_2})=K(\RR_{h_1})=K(\RR_{u})$$
and
$$h_2(x)=\left(x-\frac{1}{t_2-t_1}\right)v(x)$$
where $v(x)\in K[x]$ is a degree $(n-2)$ polynomial with
$K(\RR_{v})=K(\RR_{u})$; in particular, $\Gal(v)=\Gal(u)=\ST_{n-2}$ or
$\A_{n-2}$. Again,  the standard substitution
$$x_1=1/(x-t_1), \ y_1=y/(x-t_1)^{g+1}$$
establishes a birational $K$-isomorphism between $C_f$ and a hyperelliptic
curve
$$C_{h_2}: y_1^2=h_2(x_1).$$
Now the result follows from Corollary \ref{helpcor} applied to
$h_2(x_1)$.
\end{proof}

\section{Proof of main results}
\label{monodromy}

We keep the notation and assumptions of Theorem \ref{mainENDO}.  In addition. let us put
$$X=J(C_f), \ S= \Spec(R).$$
Let us start to prove it. First, notice that
the equation $y^2=f(x)$ defines a (semi)stable   genus $g$ curve over $R$, whose  generic fiber is smooth while its closed fiber is an irreducible reduced curve with one double point. More precisely, there is a semistable projective (flat) $R$-curve
$$\Cc:= \Proj\ R[\XX,\YY,\ZZ]/(F(\XX,\YY, \ZZ)) \to \Spec(R)=S$$
where
\begin{multline*}
F(\XX,\YY,\ZZ)=(\ZZ^{g+1})^2 [(\YY/\ZZ)^{g+1})^2-f(\XX/\ZZ)]=\\
\YY^2-(\XX-t_1\ZZ)(\XX-t_2\ZZ) \ZZ^{2g}u(\XX/\ZZ)=\\
\YY^2-(\XX-t_1\ZZ)(\XX-t_2\ZZ)\sum_{i=0}^{2g} a_i\XX^i \ZZ^{2g-i} \in R[\XX,\YY,\ZZ],\\
\deg(\XX)=\deg(\ZZ)=1, \deg(\YY)=g+1.
\end{multline*}
The principal open (affine) subset $D_{+}(\ZZ)$ of $\Cc$ is
$$\Spec\ R[\XX/\ZZ, \YY/\ZZ^{g+1}]=\Spec\ R[x,y]/(y^2-f(x))$$
with $x=\XX/\ZZ, y=\YY/\ZZ^{g+1}$.
The generic fiber of $\Cc$ coincides with the hyperelliptic curve $C_f/K$. Its closed fiber $\Cc_k$ is a singular (reduced) absolutely irreducible curve over the residue field $k$,
whose only singularity is an ordinary double point $(\bar{\beta}: 0:1)$ where
$$\bar{\beta}: =t_1 \bmod m =t_2 \bmod m \in R/\m =k.$$
The normalization of $\Cc_k$ is a (smooth projective) hyperelliptic curve of genus $g-1$ over $k$.
This implies (see \cite[Ch. 9, Example 8 on p. 246]{Neron}) that $\Pic_{\Cc_k/k}^0$ is a (connected)  commutative algebraic $k$-group that is an extension of a $(g-1)$-dimensional abelian variety by a one-dimensional torus. On the other hand, $\Pic_{\Cc/S}^0$ is a quasi-projective smooth separated $S$-group  scheme \cite[Th. 1 on p. 252]{Neron}, whose closed fiber coincides with
$\Pic_{\Cc_k/k}^0$ while the generic fiber is $J(C_f)=X$. In particular,  $\Pic_{\Cc/S}^0$ is a semiabelian scheme,
 whose closed fiber has toric dimension one. Now let $\X \to S$ be the N\'eron model of $X$ with closed fiber $\X_k$. The generic fibers of both $\X$ and
$\Pic_{\Cc/S}^0$  coincide with $X=J(C_f)$.  Since $\Pic_{\Cc/S}^0$ is  semiabelian, it follows from \cite[Ch. 7, Prop. 3 on p. 182]{Neron} that the identity
components of the closed fibers of $\X$ and $\Pic_{\Cc/S}^0$ are isomorphic.
This implies that (connected) $\Pic_{\Cc_k/k}^0$ is isomorphic to the identity component  $\X_k^0$ of $\X_k$;  in particular,
${\X_k}^0$ has toric dimension one.
Now Theorem \ref{twodim} tells us that $\End(J(C_f))=\End(X)=\Z$. This proves the first assertion of Theorem \ref{mainENDO}.
The second assertion follows from Theorem \ref{tateToricOne} combined with Remark \ref{TateHodge}(i) while the third one follows from Corollary \ref{corToricOne}.

\section{Cyclotomic extensions}
\label{cycloCHAR0}

Throughout this section, $k$ is a field and $K \supset k$ its overfield that is finitely generated over $k$.

It seems that the following two lemmas are well known but I did not find a suitable reference.

\begin{lem}
 \label{const}
Let $k^{\prime}$ be the algebraic closure of $k$ in $K$. Then $[k^{\prime}:k]<\infty$, i.e., the field $k^{\prime}$ is a finite algebraic extension of $k$.
\end{lem}

\begin{proof}
The following elementary proof of Lemma \ref{const} was suggested by the referee. (My original proof was based on a
 theorem of Emmy Noether (\cite[Ch. IV, Sect. 4.2, Th. 4.14 on p. 127]{Eisenbud}) and used \cite[Ch. IV, Sect. 4.4, Prop. 4.15 on p, 129 and Cor. 4.17 on p. 131]{Eisenbud}).

Let $m$ be the transcendence degree of $K$ over $k$.
If $m=0$ then $K$ is algebraic over $k$ and the assertion is trivial. So, we may assume that $m\ge 1$.
 Let $\{x_1, . . . , x_m\} \subset K$
be a transcendental basis of $K$ over $k$ and let $K_1 := k(x_1, . . . , x_m)\subset K$ be the corresponding purely transcendental extension of $k$. Since $K$ is finitely generated over $k$ and algebraic over $K_1$, the degree $[K:K_1]$ is finite. Let us consider the compositum $k^{\prime} K_1 \subset K$  of $k^{\prime}$ and $K_1$. Since
$$K_1 \subset k^{\prime} K_1 \subset K,$$
the field $ k^{\prime} K_1$ has finite degree over $K_1$ (and $[k^{\prime}K_1:K_1]$ divides $[K:K_1]$).

 As algebraic
extensions and purely transcendental extensions are linearly disjoint,  $k^{\prime} K_1$  is isomorphic to $k^{\prime}\otimes_k K_1$ and hence
the (finite) degree $[k^{\prime}K_1:K_1] = [k^{\prime}: k]$. It follows that $k^{\prime}/k$ is also a finite algebraic field  extension.
\end{proof}

\begin{rem}
\label{separable}
The field $K$ is finitely generated over $k$ and therefore over
 $k^{\prime}$.
Suppose that $k$ is perfect. Since $k^{\prime}/k$ is finite algebraic, $k^{\prime}$ is also perfect. Since the perfect  (sub)field $k^{\prime}$ is algebraically closed in $K$, the field $K$ is separable over $k^{\prime}$ (see \cite[Appendix A1, Sect. A1.2 and Cor. A1.7 on p. 568]{Eisenbud}).
\end{rem}

\begin{lem}
\label{tensor}
Suppose $k$ is perfect. Let $\kappa/k^{\prime}$ be an algebraic field extension of finite degree.
Then $K\otimes_{k^{\prime}}\kappa$ is a field  and the field extension $(K\otimes_{k^{\prime}}\kappa)/K$ has degree $[\kappa:k^{\prime}]$. In particular, if $\kappa/k^{\prime}$ is a Galois extension
then $(K\otimes_{k^{\prime}}\kappa)/K$ is also a Galois extension and the natural map
$$\Gal(\kappa/k^{\prime}) \to \Gal((K\otimes_{k^{\prime}}\kappa)/K), \ \sigma \mapsto \{x\otimes \beta \mapsto x\otimes \sigma(\beta)\} $$
$$\ \forall \sigma \in \Gal(\kappa/k^{\prime}), x\in K, \beta \in \kappa$$
is an isomorphism of Galois groups.
\end{lem}

\begin{proof}
By Remark \ref{separable}, $K$ is separable over $k^{\prime}$.
By Exercise A.1.2a and its solution in \cite[pp. 568--569 and p. 749]{Eisenbud}  (applied to $R=K$ and $S=\kappa$) the tensor product
$K\otimes_{k^{\prime}}\kappa$ is a domain and therefore is a field, since it is a finite-dimensional $K$-algebra, whose dimension equals  $[\kappa:k^{\prime}]$.
\end{proof}

Lemma \ref{tensor} implies readily the following statement.

\begin{cor}
\label{GaloisPrime}
Suppose that $k$ is perfect and let us fix an algebraic closure $\overline{k^{\prime}}$ of $k^{\prime}$. Then $K\otimes_{k^{\prime}}\overline{k^{\prime}}$ is a field that is a Galois extension of $K=K\otimes 1$ and the Galois group $\Gal((K\otimes_{k^{\prime}}\overline{k^{\prime}})/K)$ is canonically isomorphic to the absolute Galois group $\Gal(\overline{k^{\prime}})=\Gal(\overline{k^{\prime}}/k^{\prime})$ of $k^{\prime}$.
\end{cor}

\begin{proof}[Proof of Theorem \ref{cyclotomy}]

The field $K$ is finitely generated over $\Q$. It follows from Lemma \ref{const} that the algebraic closure $\Q^{\prime}$ of $\Q$ in $K$ is an algebraic number field of finite degree
over $\Q$. Let us put $k=\Q^{\prime}$. Then $k$ is algebraically closed in $K$.  For all but finitely many primes $\ell$ the field extension
is unramified at all prime divisors of $\ell$. This implies that the ramification index of the field extension $k(\mu_{\ell^j})/k$ is, at least
$\varphi(\ell^j)=[\Q(\mu_{\ell^j}):\Q]$ at all prime divisor of $\ell$. (Here $\varphi$ is the Euler function.) This implies that $[k(\mu_{\ell^j}):k]=[\Q(\mu_{\ell^j}):\Q]$,
i.e., $k$ and $\Q(\mu_{\ell^j})$ are linearly disjoint
over $\Q$. By Lemma \ref{tensor}, $K\otimes_k k(\mu_{\ell^j})$ is a field that is an extension of $K$ of of degree $\varphi(\ell^j)$.
It follows that the natural surjective homomorphism of $k(\mu_{\ell^j})$-algebras
$K\otimes_k k(\mu_{\ell^j})\to K(\mu_{\ell^j})$ is injective and therefore is a field isomorphism. In particular, $K(\mu_{\ell^j})$ is a degree $\varphi(\ell^j)$ Galois extension
of $K$ and
$$\Gal(K(\mu_{\ell^j})/K)=\Gal(k(\mu_{\ell^j})/k)=\Gal(\Q(\mu_{\ell^j})/\Q)=(\Z/\ell^j\Z)^{*}.$$
\end{proof}

\begin{sect}
Now let us assume that $k$ is the prime finite field $\F_p$ of characteristic $p$. It follows from Lemma \ref{const} that $k^{\prime}$ is a finite field of characteristic $p$ and therefore the number $q^{\prime}=\#(k^{\prime})$   of its elements is a power of $p$. For every prime $\ell \ne p$ we write $N_p(\ell)$ (resp. $N^{\prime}(\ell)$)
the index in $(\Z/\ell\Z)^{*}$ of the cyclic multiplicative subgroup generated by $p \bmod \ell$ (resp. $q^{\prime}  \bmod \ell$). Clearly, $N_p(\ell)$ divides $N^{\prime}(\ell)$.
\end{sect}

The following assertion that is based on results of P. Moree \cite{Moree} will be proven at the end of this Section.

\begin{lem}
\label{indexUNBOUND}
The function $\ell \mapsto N_p(\ell)$ is an unbounded function in $\ell$.
\end{lem}

\begin{thm}
\label{unboundK}
\begin{itemize}
\item[(i)]
for all primes $\ell \ne p$ the image
$$\bar{\chi}_{\ell,K}(\Gal(K))\subset (\Z/\ell\Z)^{*}$$
is the cyclic multiplicative  subgroup generated by $q^{\prime} \bmod \ell$.
\item[(ii)]
Let $N_K(\ell)$ be the index $[(\Z/\ell\Z)^{*}:\bar{\chi}_{\ell}(\Gal(K))]$. Then the function $\ell \mapsto N_K(\ell)$ is an unbounded function in $\ell$.
\end{itemize}
\end{thm}

\begin{proof}[Proof of Theorem \ref{unboundK} (modulo Lemma \ref{indexUNBOUND})]
Since $k^{\prime}\subset \bar{K}$, the algebraic closure of $k^{\prime}$ in $\bar{K}$ is an  algebraically closed field and will be denoted by $\overline{k^{\prime}}$.
It follows from Corollary \ref{GaloisPrime} that there is the natural continuous surjective group homomorphism of absolute Galois groups
$$\rest: \Gal(K) \twoheadrightarrow \Gal(k^{\prime})$$
where for each automorphism $\sigma$ of $\bar{K}/K$ we write $\rest(\sigma)$ for its restriction to $\overline{k^{\prime}}$.  We need to distinguish between two cyclotomic characters
$$\bar{\chi}_{\ell,K}: \Gal(K) \to (\Z/\ell\Z)^{*}$$
and
$$\bar{\chi}_{\ell,k^{\prime}}: \Gal(k^{\prime}) \to (\Z/\ell\Z)^{*}$$
that define the action on $\ell$th roots of unity of  $\Gal(K)$ and $\Gal(k^{\prime})$ respectively. However, since all $\ell$th roots of unity of $\bar{K}$ lie in $\overline{k^{\prime}}$,
$$\bar{\chi}_{\ell,K}=\bar{\chi}_{\ell,k^{\prime}}\ \rest: \Gal(K) \twoheadrightarrow   \Gal(k^{\prime}) \to (\Z/\ell\Z)^{*} ;$$
in particular, both cyclotomic characters have the same image in $ (\Z/\ell\Z)^{*}$. Since $\Gal(k^{\prime})$ is generated (as the topological group) by the Frobenius automorphism that sends every element of $\overline{k^{\prime}}$ (including all $\ell$th roots of unity) to its $q^{\prime}$th power, the image
$$\bar{\chi}_{\ell,,k^{\prime}}(\Gal(k^{\prime}))\subset (\Z/\ell\Z)^{*}$$
is the cyclic multiplicative  subgroup generated by $q^{\prime} \bmod \ell$. It follows that the image
$$\bar{\chi}_{\ell}(\Gal(K))\subset (\Z/\ell\Z)^{*}$$
is the cyclic multiplicative  subgroup generated by $q^{\prime} \bmod \ell$,
 i.e., we proved the first assertion of our Theorem. In particular, $N_K(\ell)$ coincides with $N^{\prime}(\ell)$. Recall  that $N^{\prime}(\ell)$ is a positive integer that is divisible by 
$N_p(\ell)$.
 It follows from Lemma \ref{indexUNBOUND}
that the function
$$\ell \mapsto N^{\prime}(\ell)=N_K(\ell)$$
is an unbounded function in $\ell$.
\end{proof}

\begin{proof}[Proof of Lemma \ref{indexUNBOUND}]
Applying Lemma 4 of Section 2 in   \cite{Moree} (to $g=p$), we conclude that for every positive integer $t$ the set of primes $\ell$ such that  $t$ divides $N_p(\ell)$ is infinite. (Actually,  it is proven in \cite{Moree} that this set of primes  has a positive density.) In particular, for each $t$ there is a prime $\ell \ne p$ with $N_p(\ell) \ge t$.  This means that the function $\ell \mapsto N_p(\ell)$ is unbounded.
\end{proof}

\section{Nonisomorphic hyperelliptic curves and  jacobians}
\label{fractional}
We start to prove Theorem \ref{distinctJAC}. Replacing $K$ by its {\sl perfectization}, we may and will assume that $K$ is a perfect field.

It is well known (\cite[Ch. 2, Sect. 3, pp. 253--255]{GH}, \cite[Ch. VIII,
Sect. 3]{Dolgachev}) that the hyperelliptic curves $C_{f_1}$ and $C_{f_2}$ are
isomorphic over $\bar{K}$ if and only if there exists a fractional linear
transformation $T\in \PGL_2(\bar{K})=\Aut(\PP^1)$ that sends the branch points
of the canonical double cover $C_{f_1}\to \PP^1$ to the branch points of the
canonical double cover $C_{f_2}\to \PP^1$. If $\RR\subset \bar{K}$ is the set of roots of $u(x)$ then the corresponding sets of branch points are the disjoint unions $\RR \cup B_1$ and $\RR\cup B_2$ respectively.

Assume that $J(C_{f_1})$ and $J(C_{f_2})$ are isomorphic over $\bar{K}$. We need to get a contradiction. We know that $\End(J(C_{f_1}))=\Z$ and $\End(J(C_{f_2})))=\Z$. This implies that
both jacobians $J(C_{f_1})$ and $J(C_{f_2})$ have exactly one principal
polarization and therefore a $\bar{K}$-isomorphism  of abelian varieties
$J(C_{f_1})\cong J(C_{f_2})$ respects the principal polarizations. Now the
Torelli theorem implies that the hyperelliptic curves $C_{f_1}$ and $C_{f_2}$
 are isomorphic over $\bar{K}$. Therefore there exists a fractional linear
transformation $T\in \PGL_2(\bar{K})=\Aut(\PP^1)$ such that
$$T(\RR \cup B_1)=\RR\cup B_2.$$
Suppose that $T$ is defined over $K$, i.e., lies in $\PGL_2(K)$. Then $T$     commutes with the Galois action on $\bar{K}$ and therefore sends every Galois orbit in $\bar{K}$ onto another Galois orbit.  This implies that
$T$ sends into itself the $2g$-element Galois orbit $\RR$;
in addition, $T(B_1)=B_2$. Since
$$\Alt(\RR)\subset \Gal(u)\subset \Perm(\RR)$$  and the only permutation of $\RR$ that commutes with all even permutations is the identity map, we conclude that $T$ acts as the identity map on $\RR$. Since the number of elements in $\RR$ is greater or equal than $2g\ge 4>3$, we conclude that $T$ is the identity element of $\PGL_2(\bar{K})$ and therefore
$B_2=T(B_1)=B_1$, which is not the case. We obtained a contradiction but only under an additional assumption that $T$ lies  in $\PGL_2(K)$. Now assume that $T$ does {\sl not}  lie in $\PGL_2(K)$. It follows from Hilbert's Theorem 90 that there is a Galois automorphism $\sigma \in \Gal(K)$ such that $\sigma(T)\ne T$. On the other hand, since both sets
$\RR \cup B_1$ and $\RR \cup B_2$ are Galois-invariant,
$$\sigma(T)(\RR \cup B_1)=\RR\cup B_2.$$
If we put $U:=T^{-1}\sigma(T)\in \PGL_2(\bar{K})$ then $U$ does {\sl not} coincide with the identity automorphism of $\PP^1$ but $U(\RR \cup B_1)=\RR \cup B_1$. This implies that $U$ gives rise to a nontrivial automorphism of $C_{f_1}$ that is {\sl not} the {\sl hyperelliptic involution}. By functoriality, we obtain an automorphism of the abelian variety $J(C_{f_1})$ that is neither $1$ nor $-1$. This gives us a contradiction, because
$$\Aut(J(C_{f_1}))=\End(J(C_{f_1}))^{*}=\Z^{*}=\{\pm 1\}.$$
This ends the proof of Theorem \ref{distinctJAC}.

\section{Concluding remarks}
\label{conclude}
Let $K$ and $f(x)$ be as in Theorem \ref{mainENDO}. Let us put $X=J(C_f)$. We know that  $\End(X)=\Z$ and $X$ has somewhere a semistable reduction with toric dimension one.

Now assume that $K$ is finitely generated over its prime subfield and let  $\ell$ be a prime different from $\fchar(K)$. It follows from arguments of Remark \ref{TateHodge}(ii)
that
for every finite algebraic field extension $K^{\prime}/K$  and each
self-product $X^m$ of $X$ every $\ell$-adic Tate class on $X^m$ can be presented as a linear combination of products of divisor classes on $X^m$. In particular, the Tate conjecture
 holds true for all $X^m$ in all codimensions.

Assume additionally that $\fchar(K)=0$  and fix an embedding $\bar{K}\subset \C$. Then  arguments of Remark \ref{TateHodge}(ii) imply that for each
self-product $X^m$ of $X$ every Hodge class on $X^m$ can be presented as a linear combination of products of divisor classes on $X^m$. In particular, the Hodge conjecture
 holds true for every $X^m$ in all codimensions. In addition, the Mumford-Tate conjecture holds true for $X$.

\end{document}